\documentclass
[a4paper,10pt]{article}

\usepackage[british,american]{babel} 
\usepackage{amssymb}
\usepackage{amsmath}
\usepackage{amsthm} 
\usepackage{fancyhdr}	
\usepackage[usenames]{color}	 
\usepackage{enumerate}		 
\usepackage[latin1]{inputenc}
\usepackage[colorlinks,linkcolor=blue,citecolor=blue,
pagebackref]{hyperref}

\fancypagestyle{main}{
 \fancyhf{}		

 \fancyhead[C]{\scriptsize Differential equations in Colombeau generalized 
functions}
 \fancyhead[R]{}

 \fancyfoot[L]{\scriptsize E.~E. \& M.~G.}
 \fancyfoot[C]{\thepage}
 \fancyfoot[R]{\scriptsize \today}
 
}

\fancypagestyle{kapitelanfang}{
 \fancyhead[L]{}
 \fancyhead[C]{}
 \fancyhead[R]{}
 \fancyfoot[L]{\scriptsize E.~E. \& M.~G.}
 \fancyfoot[C]{\thepage}
 \fancyfoot[R]{\scriptsize \today}
 
}

\newcommand{\ve}{\varepsilon}

\newcommand{\ee}{_\varepsilon}

\newcommand{\setN}{\mathbb{N}}

\newcommand{\setR}{\mathbb{R}}

\newcommand{\calE}{{\mathcal E}}
\newcommand{\calG}{{\mathcal G}}
\newcommand{\calN}{{\mathcal N}}
\newcommand{\calO}{{\mathcal O}}

\newcommand{\calR}{{\mathcal R}}

\newcommand{\Diff}{\mathrm{D}}	
\DeclareMathOperator{\id}{id}

\newcommand{\Con}{{\mathcal C}}

\DeclareMathOperator{\dist}{dist}

\newcommand{\flip}{^\mathrm{fl}}

\newcommand{\toe}[1]{\ti{#1}_{0 \ve}}
\newcommand{\tee}[1]{\ti{#1}_{1 \ve}}

\newcommand{\comp}{\subset\subset} 
\newcommand{\cptin}{\subset\subset}
\newcommand{\bs}{\backslash}
\newcommand{\epsint}{{(0,1]}}

\newcommand{\eps}{\varepsilon}
\newcommand{\cc}{\mathcal{C}}
\newcommand{\R}{\mathbb R}
\newcommand{\cg}{\mathcal{G}}
\newcommand{\cn}{\mathcal{N}}
\newcommand{\ti}{\widetilde}

\newcommand{\Cinf}{{\cc^\infty}}
\newcommand{\al}{\alpha}
\newcommand{\N}{\mathbb N}
\newcommand{\pa}{\partial}

\newcommand{\bet}{\beta}
\newcommand{\ga}{\gamma}
\newcommand{\de}{\delta}

\newcommand{\diff}{\mathrm{d}}

\newcommand{\Leb}{\mathrm{L}}
\newcommand{\Jd}{J^{\dagger}}

\newtheoremstyle{thm}{\topsep}{\topsep}%
     {\slshape}
     {}
     {\bfseries}
     {.}
     { }
     {\thmnumber{#2.\,\,}\thmname{#1}\thmnote{\;\,(#3)}}

\newtheoremstyle{exi}{\topsep}{\topsep}%
     {}
     {}
     {\bfseries}
     {.}
     { }
     {\thmnumber{#2.\,\,}\thmname{#1}\thmnote{\;\,(#3)}}

\theoremstyle{thm}
\newtheorem{theorem}{Theorem}[section]
\newtheorem{lemma}[theorem]{Lemma}
\newtheorem{proposition}[theorem]{Proposition}

\newtheorem{definition}[theorem]{Definition}

\theoremstyle{exi}
\newtheorem{example}[theorem]{Example}
\newtheorem{remark}[theorem]{Remark}

\begin{document}

\title{Ordinary differential equations in algebras of generalized functions}

\author{Evelina Erlacher\thanks{
Institute for Statistics and Mathematics, 
Vienna University of Economics and Business, 
Augasse 2-6, 1090 Vienna, 
Austria, 
\hbox{e-mail:} evelina.erlacher@wu.ac.at
}\;,\ \ 
Michael Grosser\thanks{
Faculty of Mathematics,
University of Vienna,
Oskar-Morgenstern-Platz 1, 1090 Vienna,
Austria,
\hbox{e-mail:} michael.grosser@univie.ac.at}
}


\maketitle

\begin{abstract}
A local existence and uniqueness theorem for ODEs in the special 
algebra of generalized functions is established, as well as versions including
parameters and dependence on initial values in the generalized sense. Finally, a
Frobenius theorem is proved. In all these results, composition of generalized
functions is based on the notion of c-boundedness.
\\[-0\baselineskip]

\setlength{\parindent}{0pt}
\end{abstract}



\vspace*{.3em}\section{Introduction}

At the time of their introduction in the 1980s (\cite{c1}, \cite{c2}), algebras 
of generalized functions in the Colombeau setting were primarily intended as a 
tool for treating nonlinear (partial) differential equations in the presence of 
singularities. Since then, many types of differential equations have been 
studied in the Colombeau setting (see \cite{MObook}, together with the 
references given therein, and the first part of \cite{procESI} for a variety of 
examples). Nevertheless, the authors of \cite{GKOS01} feel compelled to declare 
some 15 years later that ``a refined theory of local solutions of ODEs is not 
yet fully developed'' (p.\ 80). In fact, this state of affairs has not changed 
much since then. It is the purpose of this article to lay the foundations for 
such a theory, with composition of generalized functions based on the concept of 
c-boundedness.

As the basic object of study one may view the differential equation $\dot u(t) 
= F(t,u(t))$ with initial condition $u(\ti t_0)= \ti x_0$. Since $u(t)$ gets 
plugged into the second slot of $F$ it is evident that one has to adopt a 
suitable concept of composition of generalized functions in order to give 
meaning to the right-hand side of the ODE, keeping in mind that in general, the 
composition of generalized functions is not defined.

One way of handling the composition $u\circ v$ of generalized functions $u,v$ 
is to assume the left member $u$ to be tempered (see \cite[Subsection 
1.2.3]{GKOS01} for a definition). In this setting, a number of results on ODEs 
have been established, including a global existence and uniqueness theorem 
(\cite[Theorem 3.1]{her99}, \cite[Theorem 1.5.2]{GKOS01}). A more recent concept 
of composing generalized functions goes back to Aragona and Biagoni (cf.\ 
\cite{AB}): Here, the right member $v$ is assumed to be {\it compactly bounded} 
({\it c-bounded}\,) into the domain of $u$ (see Section \ref{prelim} for 
details); then the composition $u\circ v$ is defined as a generalized function. 
It is this latter approach we will adopt in this article. It seems to be suited 
better to local questions; moreover, the concept of c-boundedness permits an 
intrinsic generalization to smooth manifolds (\cite[Subsection 3.2.4]{GKOS01}, 
contrary to that of tempered generalized distributions.

In a number of contributions, the notion of c-boundedness has already been taken 
as the basis for the treatment of generalized ODEs. The first instance, dating 
back to  \cite{kun99b}, served as a tool for an application to a problem in 
general relativity, see \cite[Lemma 5.3.1]{GKOS01} and the improved version in 
\cite[Lemma 4.2]{EGpenrose}. Theorem 3.1 of \cite{flow}---where a theory of 
singular ordinary differential equations on differentiable manifolds is 
developed---provides a global existence and uniqueness result for autonomous 
ODEs on $\R^n$. Theorem 1.9 in \cite{hahoe08} establishes existence of a 
solution assuming an $\Leb^1$-bound (as a function of $t$, uniformly on $\R^n$ 
with respect to the second slot) on the representatives of $F$. Finally, the 
study of the Hamilton-Jacobi equation in the framework of generalized functions 
in \cite{fer11} led to some local existence and uniqueness results for ODEs, in 
a setting adapted to this particular problem. We will discuss
one of these Theorems in more detail in Section \ref{ode}.

A special feature of the existence and uniqueness results \ref{gen ode} and 
\ref{gen ode p} in Section \ref{ode} consists in their capacity to 
simultaneously allow generalized values both for $\ti t_0$ and $\ti x_0$ in the 
initial conditions, and to have, nevertheless, the domain of existence of the 
local solution equal to the one in the classical case.

The results of this article may be viewed as extending and refining the material 
of Chapter 5 of \cite{EEdiss}. Section \ref{prelim} makes available the 
necessary technical prerequisites. Local existence and uniqueness results for 
ODEs in the c-bounded setting are the focus of Section \ref{ode}: Following the 
basic theorem handling the initial value problem mentioned above, two more 
statements are established covering ODEs with parameters and $\cg$-dependence of 
the solution on initial values, respectively. Section \ref{frob}, finally, 
presents a generalized version of the theorem of Frobenius, also in the 
c-bounded setting.


\section{Notation and preliminaries}
\label{prelim}

For subsets $A,B$ of a topological space $X$, we write $A\subset\subset B$ if 
$A$ is a compact subset of the interior $B^\circ$ of $B$. By $B_r(x)$ we denote 
the open ball with centre $x$ and radius $r>0$. We will make free use of the 
exponential law and the argument swap (flip), i.e.\ for functions $f: X \times Y 
\to Z$ we will write $f(x)(y)=f(x,y)=f\flip (y,x)=f\flip (y)(x)$.

Generally, the special Colomeau algebra can be constructed with real-valued or 
with complex-valued functions. For the purposes of the present article we 
consider the real version only. Concerning fundamentals of (special) Colombeau 
algebras, we follow \cite[Subsection 1.2]{GKOS01}.

In particular, for defining the special Colombeau algebra $\calG (U)$ on a 
given (non-empty) open subset $U$ of $\setR^n$, we set $\calE (U ) 
:=\Con^\infty (U,\setR)^\epsint$ and
\begin{align*}
\calE_M (U ) & := \{ (u\ee)\ee \in \calE (U) \, | \, \forall \, K \cptin U \ 
\forall \, \alpha \in \setN_0^n \ \exists \, N \in \setN : \\
& \phantom{aaaaaaaaaaaaaaaaaaaaaaaaa} \, \sup_{x \in K} | \partial^\alpha u\ee 
(x) | = O(\ve^{-N}) \mbox{ as } \ve \to 0 \}, \\
\calN (U ) & := \{ (u\ee)\ee \in \calE (U) \, | \, \forall \, K \cptin U \ 
\forall \, \alpha \in \setN_0^n \ \forall \, m \in \setN : \\
& \phantom{aaaaaaaaaaaaaaaaaaaaaaaaa} \, \sup_{x \in K} | \partial^\alpha u\ee 
(x) | = O(\ve^m) \mbox{ as } \ve \to 0 \}.
\end{align*}
Elements of\/ $\calE_M(U)$ and $\calN(U)$ are called {\it moderate} and {\it 
negligible functions}, respectively. By \cite[Theorem 1.2.3]{GKOS01}, 
$(u\ee)\ee$ is already an element of $\calN (U)$ if the above conditions are 
satisfied for $\al=0$. $\calE_M(U)$ is a subalgebra of $\calE (U)$, $\calN
(U)$ is an ideal in $\calE_M (U)$. The {\it special Colombeau algebra} on $U$ 
is defined as
\[
\calG (U) := \calE_M (U) / \calN (U).
\]

The class of a moderate net $(u\ee)\ee$ in this quotient space will be denoted 
by $[(u\ee)\ee]$. A generalized function on some open subset $U$ of $\setR^n$ 
with values in $\setR^m$ is given as an $m$-tuple $(u_1,\cdots,u_m)\in 
\calG(U)^m$ of generalized functions $u_j\in\calG(U)$ where $j=1,\cdots,m$.

$U \to \calG (U)$ is a fine sheaf of differential algebras on $\setR^n$.

The composition $v\circ u$ of two arbitrary generalized functions is not 
defined, not even if $v$ is defined on the whole of $\setR^m$ (i.e., if 
$u\in\calG(U)^m$ and $v\in\calG(\setR^m)^l$). A convenient condition for $v\circ 
u$ to be defined is to require $u$ to be ``compactly bounded'' (c-bounded) into 
the domain of $v$. Since there is a certain inconsistency in \cite{GKOS01} 
concerning the precise description of c-boundedness (see \cite[Section 
2]{EEinvers} for details) we
include the explicit definition of this important property below. For a full 
discussion, see again \cite[Section 2]{EEinvers}.

\begin{definition}
\label{cbounded neu}
Let $U$ and $V$ be open subsets of $\setR^n$ and $\setR^m$, respectively.
\begin{enumerate}[(1)]
\item
An element $(u\ee)\ee$ of $\calE_M (U)^m$ is called {\it c-bounded from $U$ 
into $V$} if the following conditions are satisfied:
\begin{enumerate}[(i)]
\item
There exists $\ve_0 \in \epsint$, such that $u\ee(U) \subseteq V$ for all $\ve 
\le \ve_0$.
\item
For every $K \cptin U$ there exist $L \cptin V$ and $\ve_0 \in \epsint$ such 
that $u\ee (K) \subseteq L$ for all $\ve \le \ve_0$.
\end{enumerate}
The collection of c-bounded elements of $\calE_M (U)^m$ is denoted by $\calE_M 
[U,V]$.
\item
An element $u$ of $\calG (U)^m$ is called {\it c-bounded from $U$ into $V$} if 
it has a representative which is c-bounded from $U$ into $V$. The space of all 
c-bounded generalized functions from $U$ into $V$ will be denoted by $\calG 
[U,V]$.
\end{enumerate}
\end{definition}

\begin{proposition}
\label{composition}
Let $u \in \calG (U)^m$ be c-bounded into $V$ and let $v \in \calG (V)^l$, 
with representatives $(u\ee)\ee$ and $(v\ee)\ee$, respectively. Then the 
composition
\[
v \circ u := [(v\ee \circ u\ee)\ee]
\]
is a well-defined generalized function in $\calG (U)^l$.
\end{proposition}

Generalized functions can be composed with smooth classical functions provided 
they do not grow ``too fast'':
The space of {\it slowly increasing smooth functions} is given by
\begin{align*}
\calO_M (\setR^n):= \{ f \in \Con^\infty (\setR^n) \, | \, \forall \, \alpha 
\in \setN_0^n \ & \exists \, N \in \setN_0 \ \exists \, C>0: \\
& \, | \partial^\alpha f(x) | \le C (1+|x|)^N \ \forall \, x \in \setR^n \} .
\end{align*}

\begin{proposition}
\label{composition O_M}
If $u=[(u\ee)\ee] \in \calG (U)^m$ and $v \in \calO_M (\setR^m)$, then
\[
v \circ u := [(v \circ u\ee)\ee] 
\]
is a well-defined generalized function in $\calG (U)$.
\end{proposition}

We call $\calR := \calE_M / \calN$ the ring of {\it generalized numbers}, where
\begin{align*}
\calE_M & := \{ (r\ee)\ee \in \setR^\epsint \, | \, \exists \, N \in \setN : \ 
| r\ee | = O(\ve^{-N}) \mbox{ as } \ve \to 0 \}, \\
\calN & := \{ (r\ee)\ee \in \setR^\epsint \, | \, \forall \, m \in \setN : \ | 
r\ee | = O(\ve^m) \mbox{ as } \ve \to 0 \}.
\end{align*}
For $u:=[(u\ee)\ee] \in \calG(U)$ and $x_0 \in U$, the {\it point value of $u$ 
at $x_0$} is defined as the class of $(u\ee(x_0))\ee$ in $\calR$.

On
\[
U_M := \{ (x\ee)\ee \in U^\epsint \, | \, \exists \, N \in \setN : \ | x\ee | 
= O(\ve^{-N}) \mbox{ as } \ve \to 0 \}
\]
we introduce an equivalence relation by
\[
(x\ee)\ee \sim (y\ee)\ee \; \Leftrightarrow \; \forall \, m \in \setN : \ 
|x\ee-y\ee| = O(\ve^m) \mbox{ as } \ve \to 0
\]
and denote by $\ti U:=U_M / \!\! \sim$ the set of {\it generalized points}. For 
$U=\setR$ we have $\ti \setR = \calR$. Thus, we have the canonical 
identification $\widetilde{\setR^n} = \ti \setR^n = \calR^n$.

The set of {\it compactly supported points} is
\[
\ti U_c := \{ \ti x=[(\ti x\ee)\ee] \in \ti U \, | \, \exists \, K \cptin U \ 
\exists \, \ve_0 \in \epsint \ \forall \, \ve \le \ve_0: \ x\ee \in K \}.
\]
Obviously, for $u \in \calG (U)$ and $\ti x \in \ti U_c$, $u(\ti x)$ is a 
generalized number, the {\it generalized point value of $u$ at $\ti x$}.

A point $\ti x \in \ti U_c$ is called {\it near-standard} if there exists $x 
\in U$ such that $x\ee \to x$ as $\ve \to 0$ for one (thus, for every) 
representative $(x\ee)\ee$ of $x$. In this case we write $\ti x \approx x$.

Two generalized functions are equal in the Colombeau algebra if and only if 
their generalized point values coincide at all compactly supported points 
(\cite[Theorem 1.2.46]{GKOS01}). By \cite{sanja}, it is even sufficient to check 
the values at all near-standard points. We will need a slightly refined result 
which is easy to prove using the techniques of \cite[Theorem 1.2.46]{GKOS01} and 
\cite{sanja}:

\begin{proposition}
\label{punktweise}
Let $u \in \calG (U \times V)$. Then
\[
\mbox{$u=0$ in $\calG (U \times V)$} \quad \Leftrightarrow \quad 
\parbox[t]{7cm}{$u(\, . \, , \ti y)=0$ in $\calG (U)$ for all near-standard 
points $\ti y \in \ti V_c$.}
\]
\end{proposition}


\section{\kern-5.6pt Local \kern-1.7pt existence \kern-1.7pt and \kern-1.7pt 
uniqueness \kern-1.7pt results \kern-1.7pt for \kern-1.7pt ODEs}
\label{ode}

In the first theorem of this section we give sufficient conditions to guarantee 
a (unique) solution of the local initial value problem
\begin{equation}
\label{aufgabe}
\dot u(t)=F(t,u(t)), \quad u(\ti t_0)=\ti x_0,
\end{equation}
where $I$ is an open interval in $\setR$, $U$ an open subset of $\setR^n$, $F 
\in \calG (I \times U)^n$, $\ti t_0 \in \ti I_c$ and $\ti x_0 \in \ti U_c$.
A generalized function $u \in \calG[J,U]$ (where $J$ is some open subinterval 
of $I$) is called a (local) solution of \eqref{aufgabe} on $J$ around $\ti t_0 
\in\ti I_c$ with initial value $\ti x_0$ if the differential equation in 
\eqref{aufgabe} is satisfied in $\calG(J)^n$ and the initial condition in 
\eqref{aufgabe} is satisfied in the set $\ti U$ of generalized points.

Reflecting our decision to employ the concept of c-boundedness to ensure the 
existence of compositions, a solution on some subinterval $J$ of $I$ will be a 
c-bounded generalized function from $J$ into $U$ satisfying \eqref{aufgabe}. Due 
to the c-boundedness of $u$ the requirement for $\ti x_0$ to be compactly 
supported in fact does not constitute a restriction.

Theorem \ref{gen ode} generalizes Theorem 5.2 of \cite{EEdiss} insofar as the 
domain of existence of the local solution precisely equals the one in the 
classical case whereas the solution in \cite{EEdiss} is only defined on a 
strictly smaller interval. Moreover, the present version establishes uniqueness 
with respect to the largest sensible target space (i.e., $U$), as opposed to the 
more restricted statement in \cite{EEdiss}.

\begin{theorem}
\label{gen ode}
Let $I$ be an open subinterval of $\R$, U an open subset of $\R^n$, $\ti t_0$ a 
near-standard point in $\ti I_c$ with $\ti t_0 \approx t_0 \in I$, $\ti{x}_0 \in 
\ti{U}_c$ and $F\in \calG (I \times U)^n$.

Let $\al$ be chosen such that $[t_0-\al, t_0+\al] \cptin I$. Let $(\toe{x})\ee$ 
be a representative of $\ti{x}_0$ and $L\comp U$, $\eps_0\in(0,1]$ such that 
$\toe{x} \in L$ for all $\ve \le \ve_0$. With $\bet>0$ satisfying 
$L_\bet:=L+\overline{B_\bet(0)}\comp U$ set
\[
Q:=[t_0-\al,t_0+\al]\times L_\bet
\qquad
(\comp I \times U).
\]
Assume that $F$ has a representative $(F\ee)\ee$ satisfying
\begin{equation}
\label{condi}
\sup_{(t,x) \in Q} |F\ee(t,x)| \le a
\qquad
(\eps\leq\eps_0)
\end{equation}
for some constant $a>0$. Then the following holds:

\begin{enumerate}[(i)]
\item
\label{gen ode 1}
The initial value problem
\begin{equation}
\label{ivpg}
\dot u(t)=F(t,u(t)), \quad u(\ti t_0)=\ti{x}_0,
\end{equation}
has a solution $u \in \calG [J,W]$ where $J=(t_0-h,t_0+h)$ with 
$h=\min(\al,\frac{\bet}{a})$ and 
$W=L+B_\bet(0)$.

\item
\label{gen ode 2}
Every solution of (\ref{ivpg}) in $\calG [J,U]$ is already an element of $\calG 
[J,W]$.

\item
\label{gen ode 3}
The solution of (\ref{ivpg}) is unique in $\calG [J,U]$ if, in addition to 
\eqref{condi},
\begin{equation}
\label{condi 2}
\sup_{(t,x) \in J\times W} | \partial_2 F\ee (t,x) | = O (| \log \ve |)
\end{equation}
holds.
\end{enumerate}
\end{theorem}

\begin{proof}
Throughout the proof, it suffices to consider only values of $\eps$ not 
exceeding $\eps_0$. Moreover, we can assume without loss of generality that
\begin{equation}
\label{epsnull}
|\toe{t}-t_0|\leq\frac{h}{4} \quad \text{ holds for all }\eps\leq\eps_0.
\end{equation}

\eqref{gen ode 1} In a first step we fix $\eps$ and solve the (classical) 
initial value problem
\begin{equation}
\label{ivpclass}
\dot{u}_\eps(t)=F_\eps(t,u\ee(t)),
\quad u\ee(\toe{t})=\toe{x},
\end{equation}
on a suitable subinterval of $[t_0-h,t_0+h]$. To this end, set
\[
\de_\eps:=\sup\{|{\ti t}_{0\eps'}-t_0|\mid 0<\eps'\leq\eps\}
\quad
\mbox{and}
\quad
J_\eps:=[t_0-h+\de_\eps, t_0+h-\de_\eps],
\]
both for $\eps\leq\eps_0$; note that $\de\ee \to 0$ as $\eps\to 0$. By this 
choice, we have
$
J_\eps\subseteq[t_0-h,t_0+h].
$
Indeed, from $t\in J_\eps$ we infer $|t-\toe{t}|\leq |t-t_0|+|t_0-\toe{t}|\leq 
h-\de_\eps+\de_\eps.$ The solution $u_\eps$ of (\ref{ivpclass}) now is obtained 
as the fixed point of the operator $T_\eps:X\ee\to X\ee$ defined by
\[
(T_\eps f)(t):=\toe{x}+\int_{\toe{t}}^t F_\eps(s,f(s)) \, \diff s
\qquad
(t\in J_\eps)
\] where $X\ee:=\{f:J_\eps\to L_\bet\mid f\text{ is continuous}\}$ becomes a 
complete metric space when being equipped with the metric $d(f,g):=\Vert 
f-g\Vert_\infty=\sup_{t\in J_\eps}|f(t)-g(t)|$. That $T_\eps$ in fact maps 
$X\ee$ into $X\ee$ is immediate from
\begin{equation}
\label{TXX}
|(T_\eps f)(t)-\toe{x}|\leq\left|\int_{\toe{t}}^t|F_\eps(s,f(s))| \, \diff s 
\right| \leq a\cdot|t-\toe{t}|
\end{equation}
by noting that $a\cdot|t-\toe{t}|\leq ah\leq\bet$ for $t\in J_\eps$.

Now the existence of a fixed point of $T_\eps$ (hence, of a solution of 
\eqref{ivpclass}) follows from Weissinger's fixed point theorem (\cite[\S 
1]{weiss}, \cite[I.1.6 (A5)]{gradug}) by the following argument: A variant of 
\cite[Lemma 3.2.47]{GKOS01} referring only to the second slot (see \cite[Remark 
3.12]{EEdiss} for an explicit version) yields a positive constant $\ga$ 
(depending on $\eps$) such that $|F_\eps(t,x)-F_\eps(t,y)|\leq\ga\cdot|x-y|$ for 
all $(t,x),(t,y)\in Q$. From this we derive, by induction, 
$|(T_\eps^kf)(t)-(T_\eps^kg)(t)|\leq\frac{\ga^k}{k!}(t-\toe{t})^k \Vert 
f-g\Vert_\infty$ for $t\in [\toe{t},t_0+h-\de_\eps]$ and $k\in \N_0$. The case 
of $t\in[t_0-h+\de_\eps,\toe{t}]$ being similar, we finally arrive at $\Vert 
T_\eps^kf- T_\eps^kg\Vert_\infty\leq\frac{(h\ga)^k}{k!} \Vert f-g\Vert_\infty$ 
which, due to $\sum_{k=0}^\infty\frac{(h\ga)^k}{k!}=e^{h\ga}<\infty$, suffices 
for an appeal to Weissinger's theorem. We obtain a solution $u_\eps$ 
of (\ref{ivpclass}) on $J_\eps$ taking values in $L_\bet$. Moreover, 
$u_\eps(t)\in W:=L+B_\bet(0)$ for $t\in J_\eps^\circ$ by (\ref{TXX}).

If $\de_\eps=0$ (i.e., if $t_0$ is standard) then $u_\eps$ is defined on 
$[t_0-h,t_0+h]$ and we set $\ti u_\eps:=u_\eps$; by (\ref{TXX}), $\ti 
u_\eps(J)\subseteq W$. In the case $\de_\eps>0$, Lemma \ref{ausdehn} provides 
$\ti u_\eps\in\Cinf([t_0-h,t_0+h],W)$ being equal to $u_\eps$ on $\ti 
J_\eps:=[t_0-h+2\de_\eps, t_0+h-2\de_\eps]$. In both cases, $\toe{t}\in\ti 
J_\eps$, $\ti u_\eps(\toe{t})=\toe{x}$ and $\dot{ \ti u}_\eps(t)=F_\eps(t,\ti 
u\ee(t))$ holds on $\ti J_\eps$.

In order to show that $(\ti u_\eps)_\eps$ is moderate on $J=(t_0-h,t_0+h)$ it 
suffices to establish the corresponding estimates on each $\ti J_{\eps*}$ (with 
$\eps^*\leq\eps_0$), allowing us to deal with $u_\eps$ rather than with $\ti 
u_\eps$ for all $\eps\leq\eps^*$. Thus, let $t\in\ti J_{\eps^*}$ and 
$\eps\leq\eps^*$. We have $u_\eps(t)\in L_\bet$ and $|\dot u_\eps(t)|\leq a$. 
Via the moderateness estimates for $\pa_i F_\eps$ ($i=1,2)$ we now obtain, by 
differentiating $\dot{u}_\eps(t)=F_\eps(t,u\ee(t))$, an estimate of the form
\begin{equation*}
|\ddot u\ee(t)|
\le |\partial_1 F\ee (t,u\ee(t))| + |\partial_2 F\ee (t,u\ee(t))| \cdot
|{\dot u\ee}(t)|
\le C \ve^{-N}
\end{equation*}
with constants $C>0$ and $N \in \N$  not depending on $\eps$. The estimates for 
the higher-order derivatives of $u\ee$ are now obtained inductively by 
differentiating the equation for $\ddot u_\eps$.

Concerning c-boundedness of $(\ti u_\eps)_\eps$ from $J$ into $W$ let 
$\Jd:=[t_0-h',t_0+h']$ with $\frac{h}{4}<h'<h$. For $\eps$ small enough as to 
satisfy $2\de_\eps\leq h-h'$, we have $\Jd\subseteq \ti J_\eps$.  (\ref{TXX}) 
now yields $\ti u_\eps(\Jd)=u_\eps(\Jd)\subseteq 
L+\overline{B_{a(h'+\de_\eps)}}\comp L+B_\bet(0)$.

Now that we have shown that the net $(\ti u\ee)\ee$ represents a member of 
$\calG[J,W]$ ($\subseteq \calG[J,U]$), it follows from the result established 
for fixed $\ve$ that the class of $(\ti u\ee)\ee$ is a solution of \eqref{ivpg} 
on $J$ in the sense specified at the beginning of this section: Due to the fact 
that equality in Colombeau spaces involves null estimates only on compact 
subsets of the domain, it indeed suffices that every $\ti u\ee$ satisfies the 
(classical) equation on $\ti J\ee$, taking into account $\de\ee \to 0$.

\eqref{gen ode 2} Assume that $v=[(v_\eps)_\eps]\in\cg[J,U]$ satisfies $\dot 
v(t)=F(t,v(t))$ and $v(\ti t_0)=\ti{x}_0$. With $\toe{t}$, $\toe{x}$ and 
$F_\eps$ as in part \eqref{gen ode 1} we have $v_\eps(\toe{t})=\toe{x}+\ti 
n_\eps$ and $\dot v_\eps(t)=F_\eps(t,v_\eps(t))+n_\eps(t)$ for some  
$(\ti{n}\ee)\ee \in \cn^n$ and $(n\ee)\ee \in \cn (J)^n$, respectively.

In order to show that $v\in\cg[J,W]$ with $W=L+B_\bet(0)$ we again choose 
$\Jd=[t_0-h',t_0+h']\comp J$ with $\frac{h}{4}<h'<h$. Setting 
$\de:=\frac{a}{2}(h-h')$, we select $\eps_1(\leq\eps_0)$ such that for all 
$\eps\leq\eps_1$, the three conditions $|\ti n_\eps|<\frac{\de}{3}$, 
$\int_{\Jd}|\ti n_\eps(s)| \, \diff s<\frac{\de}{3}$ and 
$a|\de_\eps|<\frac{\de}{3}$ are satisfied. Now for $\eps\leq\eps_1$, we claim 
that $|v_\eps(t)-\toe{x}|\leq\frac{a}{2}(h+h')$ holds for all $t\in 
\Jd_+:=[\toe{t},t_0+h']$. If $|v_\eps(t)-\toe{x}|<\frac{a}{2}(h+h')$ for all 
$t\in\Jd_+$, then we are done. Otherwise, choose $t^*$ minimal in $\Jd_+$ with 
$|v_\eps(t^*)-\toe{x}|=\frac{a}{2}(h+h')$. We demonstrate that, in fact, 
$t^*=t_0+h'$. From the estimate
\begin{align*}
\frac{a}{2}(h+h')
& =|v_\eps(t^*)-\toe{x}| \leq |\ti n_\eps|+\int_{\toe{t}}^{t^*}|\ti n_\eps(s)| 
\, \diff s + \int_{\toe{t}}^{t^*}|F_\eps(t,\underbrace{v_\eps(t)}_{\in L_\bet})| 
\, \diff s \\[-12pt]
& \leq \frac{\de}{3}+\frac{\de}{3}+a|\de_\eps|+a(t^*-t_0)\\
& \leq \frac{a}{2}(h-h')+a(t^*-t_0)
\end{align*}
it readily follows that $t^*\geq t_0+h'$, and thus  $t^*=t_0+h'$. Since, by a 
similar argument, $|v_\eps(t)-\toe{x}|\leq\frac{a}{2}(h+h')$ holds also for all 
$t\in \Jd_-=[t_0-h',\toe{t}]$ we finally arrive at
\[
v_\eps(\Jd)\subseteq L+\overline{B_{\frac{a}{2}(h+h')}(0)}\comp L+B_\bet(0)=W.
\]
This proves that $v$ is c-bounded from $J$ into $W$.

\eqref{gen ode 3} Let $v=[(v\ee)\ee] \in \cg [J,U]$ be another solution and 
$(n\ee)\ee \in \cn^n$, $(\ti{n}\ee)\ee \in \cn^n$ as above. By \eqref{gen ode 
2}, $v\in\cg[J,W]$. As before let $\Jd:=[t_0-h',t_0+h']$ (with 
$\frac{h}{4}<h'<h$) be a compact subinterval of $J$. Since both $(u\ee)\ee$ and  
$(v\ee)\ee$ are c-bounded from $J$ into $W$, there exists a compact subset $K$ 
of $W$ such that $u\ee(\Jd) \subseteq K$ and $v\ee(\Jd)\subseteq K$ for $\ve$ 
sufficiently small. Moreover, we can assume $\de_\eps<h-h'$. Applying the 
second-slot version of \cite[Lemma 3.2.47]{GKOS01} to the function $F_\eps$ and 
some (fixed) compact set $K'$ with $K\comp K'\comp W=L+B_\bet(0)$ yields a 
constant $C'$ (only depending on $K'$) such that
\begin{align*}
|F\ee (t,x) - F\ee (t,y)|
& \le C' \sup_{(s,z) \in \Jd \times K'} (|F\ee (s,z)| +  |\partial_2 F\ee 
(s,z)|) \cdot |x-y| \\
& \le C' ( a+ C_1|\log\eps|) \cdot|x-y|
\end{align*}
holds for all $t\in\Jd$ and all $x,y\in K$ (note that $\Jd \times K'\subseteq 
J\times W\subseteq Q$) where $C_1>0$ is the constant provided by (\ref{condi 
2}). Therefore, for $t\in \Jd$ it follows that
\begin{align*}
|v\ee(t)&-u\ee(t)| \le \\
\phantom{a} & \le |\toe{y} - \toe{x} | + \bigg| \int_{\toe{t}}^t ( |F\ee(s, 
v\ee(s)) - F\ee(s,u\ee(s) )|+|n\ee(s)| ) \, \diff s \, \bigg| \\
& \leq |\ti{n}\ee| + \bigg| \int_{\toe{t}}^t |n\ee(s)| \, \diff s \, \bigg| + 
C'(a+C_1 |\log \ve |) \cdot \bigg| \int_{\toe{t}}^t |v\ee(s)-u\ee(s)| \, \diff s 
 \, \bigg| \\
& \le C_2 \, \ve^m + (C_3 + C_4 |\log \ve|) \cdot \bigg| \int_{\toe{t}}^t 
|v\ee(s)-u\ee(s)| \, \diff s \, \bigg|
\end{align*}
for suitable constants $C_2,C_3,C_4 >0$ and arbitrary $m \in \N$. By Gronwall's 
Lemma, we obtain
\[
\sup_{t \in \Jd} |v\ee(t)-u\ee(t)|
\le C_2 \, \ve^m \cdot e^{(C_3 + C_4 |\log \ve|) \cdot | \int_{\toe{t}}^t 1 \, 
\diff s |}
\le C_0\, \ve^{m-h \, C_4}
\]
for some constant $C_0>0$ (note that $|\toe{t}-t_0|\leq h'+\de_\eps\leq h$). 
This concludes the proof of the theorem.
\end{proof}

\begin{remark}
\label{rem gen ode}
\begin{enumerate}[(i)]
\item
\label{rem gen ode 1}
The proof of Theorem \ref{gen ode} establishes the following statement on the 
level of representatives: For any given representatives $(\toe{t})\ee$ of 
$\ti{t}_0\in\ti I_c$ ($\toe{t}\to t_0\in I$), $(\toe{x})\ee$ of 
$\ti{x}_0\in\ti{U}_c$ and $(F\ee)\ee$ of $F\in\calG (I\times U)^n$ satisfying 
(\ref{condi}) the following holds:
If $\al$, $L$, $\eps_0$ and $\bet$ are chosen as in Theorem \ref{gen ode} 
(including condition (\ref{epsnull}) as to $\eps_0$), then $u$ has a 
representative $(\ti u_\eps)\ee$ that on every compact subinterval of $J$ 
satisfies the classical initial value problem \eqref{ivpclass} for $\eps$ 
sufficiently small.

\item
\label{rem gen ode 2}
If $\ti t_0$ is standard, i.e.\ (without loss of generality) $\toe{t}=t_0\in I$ 
for all $\eps$, then $\de_\eps=0$ and every $u_\eps$ exists (as a solution of 
(\ref{ivpclass})) even on $[t_0-h,t_0+h]$.

\item
\label{rem gen ode 3}
If $\ti x_0$ is standard, i.e.\ (without loss of generality) $\toe{x}=x_0\in U$ 
for all $\eps$, then $L:=\{x_0\}$ yields $L_\bet=\overline{B_\bet(x_0)}$ as in 
the classical case.
\end{enumerate}
\end{remark}

\begin{lemma}
\label{ausdehn}
\begin{enumerate}[(i)]

\item
\label{ausdehn 1}
Let $a<a_1<a_2<b_2<b_1<b$ and let $U$ be a (non-empty) open subset of $\R^n$. 
Then for $f\in\Cinf([a_1,b_1],U)$ being given, there exists $\ti 
f\in\Cinf([a,b],U)$ with $\ti f=f$ on some open neighbourhood of $[a_2,b_2]$.

\item
\label{ausdehn 2}
For any given positive $\de$, the function $\ti f$ can be chosen such as to 
satisfy $\ti f([a,b])\subseteq f([a_1,b_1])\cup B_\de(f(a_1))\cup 
B_\de(f(b_1))$.

\end{enumerate}
\end{lemma}

\begin{proof}
\eqref{ausdehn 1} Choose $\de>0$ as to satisfy 
$\overline{B_\de(f(a_1))}\cup\overline{B_\de(f(a_2))}\subseteq U$. Choose 
$\eta>0$ such that $f(t)\in B_\de(f(a_1))$ holds for $t\in [a_1,a_1+2\eta]$ and 
$f(t)\in B_\de(f(b_1))$ holds for $t\in [b_1-2\eta,b_1]$; without loss of 
generality we may assume $\eta<\frac{1}{3}\min(a_2-a_1,b_1-b_2)$.

Now let $\psi$ be a smooth function with $0\leq\psi\leq1$ such that $\psi=1$ on 
$[a_1+2\eta,b_1-2\eta]$ and $\psi=0$ outside $(a_1+\eta,b_1-\eta)$. Then $\ti f$ 
defined on $[a,b]$ by
\[
\ti f(t):=
\begin{cases}
 f(a_1) & t\in[a,a_1+\eta] \\
 f(t)\psi(t)+f(a_1)(1-\psi(t)) & t\in[a_1,a_2] \\
 f(t) & t\in[a_1+2\eta,b_1-2\eta] \\
 f(t)\psi(t)+f(b_1)(1-\psi(t)) & t\in[b_2,b_1] \\
 f(b_1) & t\in[b_1-\eta,b]
\end{cases}
\]
satisfies all requirements since each of the five defining terms is smooth and 
on overlaps the two relevant terms give rise to the same values.

\eqref{ausdehn 2} is clear from the proof of \eqref{ausdehn 1}.
\end{proof}

Theorem \ref{gen ode} is distinguished from the related result \cite[Theorem 
4.5]{fer11} by the following features: The existence statement (i) of Theorem 
\ref{gen ode} does not require logarithmic control of derivaties of $F$ which, 
by contrast, is assumed in \cite{fer11}; the domain interval of the solution in 
Theorem \ref{gen ode} equals the classical (open) one given by $(t_0-h,t_0+h)$ 
with $h=\min(\al,\frac{\bet}{a})$ while in \cite{fer11} one has to take 
$h<\min(\al,\frac{\bet}{a})$; finally, the boundedness assumption on $F$ in 
\cite{fer11} refers to the whole open domain of $F$ whereas in Theorem \ref{gen 
ode} it suffices to have boundedness of $F$ on the (compact) subset $Q$. 
Generally, all existence and uniqueness results for ODEs in \cite{fer11} are 
tailored for applications of the method of characteristics to the generalized 
Hamilton-Jacobi problem; hence the setting of \cite{fer11} always includes 
initial conditions as parameters, necessitating the logarithmic growth 
condition even for existence results (compare Theorem \ref{gen ode p} below).

The following three examples illustrate the significance of the boundedness 
assumption on $F$ by displaying increasing obstacles against obtaining a 
generalized solution from the classical ones obtained for fixed $\eps$, in the 
absence of condition \eqref{condi}.

\begin{example}
Let $F \in \calG (\setR \times \setR)$ be given by the representative $F\ee 
(t,x):= \frac 1 \ve \big( 2- \frac{1}{1+x^2} \big)$, and let $t_0=0$ and 
$x_0=0$. Then $F$ fails to satisfy condition \eqref{condi} on any neighbourhood 
of $(t_0,x_0)$. Nevertheless, there exists a unique global solution for every 
$\ve$: Integrating $\dot x(t)=F\ee(t,x)$ yields $\frac x 2 + \frac{1}{2 \sqrt 2} 
\arctan (\sqrt 2 \, x) = \frac 1 \ve t$. Setting $f(x)=\frac x 2 + \frac{1}{2 
\sqrt 2} \arctan (\sqrt 2 \, x)$, we obtain $u\ee(t):=f^{-1}(\frac 1 \ve t)$ as 
the solution of the classical initial value problem. By Proposition 
\ref{composition O_M}, $(u\ee)\ee \in \calE_M (\setR)$. However, $(u\ee)\ee$ is 
not c-bounded. Hence, $u\ee$ solves the differential equation for every $\ve$ 
but on any interval around $0$, the generalized function $[(u\ee)\ee]$ is not a 
solution of the initial value problem in the setting of the c-bounded theory of 
ODEs since the composition $F(t,u(t))$ exists only componentwise on the level 
of representatives, yet not in the sense of Proposition \ref{composition}.
\end{example}

\begin{example}
Let $F \in \calG (\setR \times \setR)$ be given by the representative $F\ee 
(t,x):= \frac x \ve$, and let $t_0=0$ and $x_0=1$. Again, $F$ does not satisfy 
condition \eqref{condi} on any neighbourhood of $(t_0,x_0)$. For each $\ve$, 
there exists a unique (even global) solution $u\ee(t)=e^{\frac t \ve}$. However, 
$(u\ee)\ee$ is not moderate on any neighbourhood of $0$.
\end{example}

\begin{example}
\label{e: schrumpf}
Let $F \in \calG \big( \setR \times (\setR \bs \{ -1 \}) \big)$ be defined by 
the representative $F\ee (t,x):= -\frac{t}{x+1} \cdot g(\ve)$ where $g: (0,1] 
\to \setR$ is a smooth map satisfying $g(\ve) \to \infty$ for $\ve \to 0$. Let 
$t_0=0$ and $x_0=0$. Then $F$ violates condition \eqref{condi} on any 
neighbourhood of $(t_0,x_0)$. For every $\ve$ we obtain (unique) solutions
$
u\ee(t)= \sqrt{1- g(\ve) \,t^2}-1
$
that are defined, at most, on the open interval
$( - 1/\sqrt{g(\ve)} , 1/\sqrt{g(\ve)} )$. Hence, there is not even a common domain.

In this example, $F$ failing to satisfy condition $\eqref{condi}$ leads to 
shrinking of the solutions' domains as $\ve \to 0$. Note that this result is not 
a consequence of the rate of growth of $|F\ee (t,x)|$ on any compact set; 
rather, it only matters that $|F\ee (t,x)|$ does increase infinitely (as $\ve 
\to 0$).
\end{example}

Theorem \ref{gen ode} can handle jumps as the following example shows.

\begin{example}
Let $I$ be an open interval in $\setR$ and $U$ an open subset of $\setR^n$. 
Consider the initial value problem
\begin{equation}
\label{ode H}
\dot u (t) = f(t,u(t)) \cdot (\iota H) (t) + g(t,u(t)), \quad u(t_0)=x_0,
\end{equation}
where $f,g \in \Con^\infty (I \times U,\setR^n)$, $t_0 \in I$, $x_0 \in U$, and 
where $\iota H$ denotes the embedding of the Heaviside function $H$ into the 
Colombeau algebra. If $\rho$ is a mollifier (i.e.\ a Schwartz function on 
$\setR$ satisfying $\int \rho(x) \, \diff x = 1$ and $\int x^\alpha \rho(x) \, 
\diff x  = 0$ for all $\alpha \ge 1$), then a representative $(H\ee)\ee$ of 
$\iota H$ is given by $H\ee (t)=\int_{-\infty}^t \frac{1}{\ve} \, \rho \left( 
\frac s \ve \right) \, \diff s$. Fix some $\alpha>0$ such that 
$[t_0-\alpha,t_0+\alpha] \subseteq I$ and choose an open subset $W$ of $U$ with 
$x_0 \in W \subseteq \overline{W} \cptin U$. A short computation shows that $| 
H\ee (t) | \le \| \rho \|_{L^1 (\setR^n)}$ for all $t$. Thus, $| f(t,x) \cdot 
H\ee (t) + g(t,x) | \le a_1 \| \rho \|_{L^1 (\setR^n)} + a_2=:a$ on 
$[t_0-\alpha,t_0+\alpha] \times \overline{W}$ for some constants $a_1,a_2>0$. 
Hence, by Theorem \ref{gen ode}, the initial value problem \eqref{ode H} 
possesses a solution $u$ in $\calG [J,W]$ where $J:= (t_0-h,t_0+h)$ and $h= \min 
\Big( \alpha, \frac{\dist (x_0,\partial W)}{a} \Big)$. Since the initial value 
problem also satisfies \eqref{condi 2}, the solution is unique in $\calG [J,U]$.
\end{example}

Next, we turn our attention to generalized ODEs including parameters. In view 
of our goal to establish a Frobenius theorem in the present setting, we want 
the solution to be $\calG$-dependent on the parameter.

\begin{theorem}
\label{gen ode p}
Let $I$ be an open subinterval of $\setR$, $U$ an open subset of $\setR^n$, $P$ 
an open subset of $\setR^l$, $\ti t_0$ a near-standard point in $\ti I_c$ with 
$\ti t_0 \approx t_0 \in I$, $\ti{x}_0 \in \ti{U}_c$ and $F \in \calG (I \times 
U \times P)^n$.

Let $\al$ be chosen such that $[t_0-\al, t_0+\al] \cptin I$. Let $(\toe{x})\ee$ 
be a representative of $\ti{x}_0$ and $L\comp U$, $\eps_0\in(0,1]$ such that 
$\toe{x} \in L$ for all $\ve \le \ve_0$. With $\bet>0$ satisfying 
$L_\bet:=L+\overline{B_\bet(0)}\comp U$ set
\[
Q:=[t_0-\al,t_0+\al]\times L_\bet
\qquad
(\comp I \times U).
\]
Assume that $F$ has a representative $(F\ee)\ee$ satisfying
\begin{equation}
\label{condi p}
\sup_{(t,x,p) \in Q \times P} |F\ee(t,x,p)| \le a
\qquad (\ve  \le \ve_0)
\end{equation}
for some constant $a>0$ and that for all compact subsets $K$ of $P$
\begin{equation}
\label{condi 2 p}
\sup_{(t,x,p) \in Q \times K} | \partial_2 F\ee (t,x,p) | = O (| \log \ve |) .
\end{equation}

Then the following holds: There exists $u \in \calG [P \times J,W]$ with 
$J:=[t_0-h,\linebreak t_0+h]$, $h = \min \big( \alpha, \frac{\beta}{a} \big)$ 
and $W=L+B_\beta(0)$ such that for all $\ti p \in \ti P_c$ the map $u(\ti p,\, . 
\,) \in \calG [J,W]$ is a solution of the initial value problem
\begin{equation*}
\dot u (t)=F(t,u(t),\ti p), \quad u(\ti t_0)=\ti{x}_0 .
\end{equation*}
The solution $u$ is unique in $\calG [P \times J,U]$.
\end{theorem}

\begin{proof}
{\it Existence}: Let $(\toe{t})\ee$ be a representative of $\ti t_0$. 
Proceeding as in the proof of Theorem \ref{gen ode}, we set $\delta\ee := \sup 
\{ |\ti t_{0 \ve'} - t_0| \,|\, 0 < \ve' \le \ve \}$ and 
$J\ee:=[t_0-h+\delta\ee,t_0+h-\delta\ee]$. For every $p \in P$ there exists a 
net of (classical) solutions $u\ee(p,\,.\,): J\ee \to L_\beta$ of the initial 
value problem
\begin{equation}
\label{ivp eps 2}
\dot u\ee (t)=F\ee(t,u\ee(t),p), \quad u\ee(\toe{t})=\toe{x} \qquad (\ve \le \ve_0),
\end{equation}
satisfying $u\ee(p,J\ee^\circ) \subseteq W$. By the classical Existence and 
Uniqueness Theorem for ODEs with parameter, the mappings $(p,t) \mapsto 
u\ee(p,t)$ are $\Con^\infty$. Lemma \ref{ausdehn} provides $\ti u_\eps \in 
\Con^\infty (P \times [t_0-h,t_0+h],W)$ being equal to $u_\eps$ on $\ti 
J_\eps:=[t_0-h+2\de_\eps, t_0+h-2\de_\eps]$.

In order to show that $(\ti u\ee)\ee$ is moderate on $J$ it again suffices to 
establish the corresponding estimates for $(u\ee)\ee$. C-boundedness of $(\ti 
u\ee)\ee$ is shown as in the proof of Theorem \ref{gen ode}.

The moderateness of $(u\ee)\ee$ will be shown in three steps: First we consider 
derivatives with respect to $t$, then only derivatives with respect to $p$ and, 
finally, mixed derivatives.

The $\calE_M$-estimates for $u\ee (p,t)$, $\partial_2 u\ee (p,t)$ and all its 
derivatives with respect to $t$ are obtained in the same way as in the proof of 
Theorem \ref{gen ode}.

Next, we consider the derivatives with respect to $p$. Differentiating the 
integral equation corresponding to the initial value problem (on the level of 
representatives) with respect to $p$ yields
\begin{equation}
\label{ivp int abl}
\partial_1 u\ee (p,t) = \int_{\toe{t}}^t \Big( \partial_2 F\ee \big( 
s,u\ee(p,s),p \big) \cdot \partial_1 u\ee(p,s) + \partial_3 F\ee \big( s, 
u\ee(p,s),p \big) \Big) \, \diff s.
\end{equation}
Let $K_1 \times K_2 \cptin P \times J$ and $(p,t) \in K_1 \times K_2$. By 
$u\ee(K_1 \times K_2) \subseteq L_\beta \cptin U$ and \eqref{condi 2 p}, we 
obtain
\[
| \partial_1 u\ee (p,t) |
\le h \, C_1 \ve^{-N_1} + \bigg| \int_{\toe{t}}^t C_2 | \log \ve | \cdot | 
\partial_1 u\ee(p,s) | \, \diff s \bigg|
\]
for constants $C_1,C_2>0$ and some fixed $N \in \setN$. By Gronwall's Lemma, it 
follows that
\[
| \partial_1 u\ee(p,t) |
\le h \, C_1 \ve^{-N_1} \cdot e^{| \int_{\toe{t}}^t C_2 | \log \ve | \, \diff s |}
\le (h \, C_1) \, \ve^{-(N_1+ h \, C_2)}.
\]
Differentiating \eqref{ivp int abl} $i-1$ times with respect to $p$ ($i \in 
\setN$) gives an integral formula for $\partial_1^i u\ee(p,t)$. Observe that in 
this formula $\partial_1^i u\ee(p,t)$ itself appears on the right-hand side 
only once, namely with $\partial_2 F\ee(s,u\ee(p,s),p)$ as coefficient, and that 
the remaining terms contain only $\partial_1$-derivatives of $u\ee$ of order 
less than $i$. Thus, we may estimate the higher-order derivatives with respect 
to $p$ inductively by differentiating equation \eqref{ivp int abl} and applying 
Gronwall's Lemma.

Finally, it remains to handle the case of mixed derivatives. For arbitrary $i 
\in \setN$ we have
\begin{equation}
\label{fast schluss}
\partial^i_1 \partial_2 \, u\ee(p,t)
= \frac{\partial^i}{\partial p^i} \frac{\partial}{\partial t} \left( \toe{x} + 
\int_{\toe{t}}^t F\ee \big( s,u\ee(p,s),p \big) \, \diff s \right)
= \frac{\partial^i}{\partial p^i} F\ee \big( t,u\ee(p,t),p \big).
\end{equation}
By carrying out the $i$-fold differentiation on the right-hand side of 
equa-\linebreak tion \eqref{fast schluss}, we obtain a polynomial expression in 
$\partial_2^k F\ee \big( t,u\ee(p,t),p \big)$, \linebreak $\partial_3^k F\ee 
\big( t,u\ee(p,t),p \big)$ and $\partial_1^k u\ee (p,t)$ for $1 \le k \le i$ all 
of which satisfy the $\calE_M$-estimates. The estimates for $\partial^i_1 
\partial^j_2 \, u\ee(p,t)$ with $j \ge 2$ are now obtained inductively by 
differentiating equation \eqref{fast schluss} with respect to $t$.

{\it Uniqueness}: By Proposition \ref{punktweise}, it suffices to show that for 
every near-standard point $\ti p \in \ti P_c$ the solution $u(\ti p,\, . \,)$ is 
unique in $\calG [J,U]$. For a fixed near-standard point $\ti p=[(\ti p\ee)\ee] 
\in \ti P_c$, condition \eqref{condi 2 p} implies the condition for uniqueness 
\eqref{condi 2} in Theorem \ref{gen ode} with respect to $G\ee(t,x):=(F\ee(\, . 
\, , \, . \, ,\ti p\ee))\ee$, yielding uniqueness of $u(\ti p,\, . \,)$ in 
$\calG [J,U]$.
\end{proof}

\begin{remark}
\label{rem gen ode p}
Similarly to Remark \ref{rem gen ode}\,\eqref{rem gen ode 1}, a corresponding 
statement on the level of representatives can be extracted from the proof of the 
preceeding theorem. Also \eqref{rem gen ode 2} and \eqref{rem gen ode 3} of 
Remark \ref{rem gen ode} apply.
\end{remark}

Requiring also $\ti x_0$ in the initial condition in Theorem \ref{gen ode p} to 
be near-standard, we even can prove $\calG$-dependence of the solution on the 
initial values.

\begin{theorem}
\label{gen ode anfang}
Let $I$ be an open subinterval of $\setR$, $U$ an open subset of $\setR^n$, $P$ 
an open subset of $\setR^l$, $\ti t_0$ a near-standard point in $\ti I_c$ with 
$\ti t_0 \approx t_0 \in I$, $\ti x_0$ a near-standard point in $\ti U_c$ with 
$\ti x_0 \approx x_0 \in U$ and $F \in \calG (I \times U \times P)^n$.

With $\alpha>0$ and $\beta>0$ satisfying $[t_0-\al,t_0+\al] \cptin I$ and 
$\overline{B_\beta (x_0)} \cptin U$, respectively, set
\[
Q
:= [t_0-\al,t_0+\al] \times \overline{B_\beta (x_0)}
\qquad
( \cptin I \times U ).
\]
Assume that $F$ has a representative $(F\ee)\ee$ satisfying
\begin{equation}
\label{condi 1 a}
\sup_{(t,x,p) \in Q \times P} |F\ee(t,x,p)| \le a
\qquad (\ve  \le \ve_0)
\end{equation}
for some constant $a>0$ and $\ve_0 \in (0,1]$ and that for all compact subsets 
$K$ of $P$
\begin{equation}
\label{condi 2 a}
\sup_{(t,x,p) \in Q \times K} | \partial_2 F\ee (t,x,p) | = O (| \log \ve |) .
\end{equation}
 
Then the following holds: For fixed $h \in \left( 0 , \min \big( \alpha, 
\frac{\beta}{a} \big) \right)$ there exist open neighbourhoods $J_1$ of $t_0$ in 
$J:= ( t_0-h,t_0+h )$ and $U_1$ of $x_0$ in $U$ and a generalized function $u 
\in \calG [J_1 \times U_1 \times P \times J,B_\gamma (x_0)]$ with $\gamma \in 
(0,\beta)$ and $\beta - \gamma>0$ sufficiently small, such that for all $(\ti 
t_1,\ti x_1,\ti p) \in \ti J_{1c} \times \ti U_{1c} \times \ti P_c$ the map 
$u(\ti t_1,\ti x_1,\ti p,\, . \,) \in \calG [J,B_\gamma (x_0)]$ is a solution of 
the initial value problem
\begin{equation}
\label{ivp anfangswert}
\dot u (t)=F(t,u(t),\ti p), \quad u(\ti t_1)=\ti x_1.
\end{equation}
The solution $u$ is unique in $\calG [ J_1 \times U_1 \times P \times J , 
B_\gamma (x_0) ]$.
\end{theorem}

\begin{proof}
{\it Existence}: The basic strategy of the proof is to consider $(\ti t_0,\ti 
x_0)$ as part of the parameter and apply Theorem \ref{gen ode p}. However, we 
will have to cope with some technicalities.

Let $(\toe{t})\ee$ and $(\toe{x})\ee$ be representatives of $\ti t_0$ and $\ti 
x_0$, respectively. From now on, we always let $\ve \le \ve_0$. Let $\lambda \in 
(0,1)$ and set
\[
\hat I:= (-\lambda \alpha,\lambda \alpha), \quad I_1:= (t_0-(1-\lambda) 
\alpha,t_0+(1-\lambda) \alpha).
\]
Choose $\mu \in \big( 0,\frac{\beta}{3} \big)$, set $\gamma := \beta - 2 \mu$ 
and define
\[
\hat U :=B_{\gamma+\mu}(0), \quad U_1:=B_{\mu}(x_0).
\]
Then
$
\hat I + I_1 = (t_0-\al,t_0+\al) \subseteq I
$
and
$
\hat U + U_1 = B_\beta (x_0) \subseteq U
$.
Hence, we may define $G\ee : \hat I \times \hat U \times (I_1 \times U_1 \times 
P) \to \setR^n$ by
\[
G\ee(t,x,(t_1,x_1,p)):=F\ee(t+t_1,x+x_1,p).
\]
Obviously, $(G\ee)\ee$ is moderate and, therefore, $G:=[(G\ee)\ee]$ is in 
$\calG (\hat I \times \hat U \times \linebreak (I_1 \times U_1 \times P))^n$. 
Now let $\delta \in (0, \lambda \alpha)$ and $\eta \in (0, \gamma - \mu)$. By 
assumptions \eqref{condi 1 a} and \eqref{condi 2 a}, we obtain $| G\ee 
(t,x,(t_1,x_1,p)) | \le a$ for all $(t,x,(t_1,x_1,p)) \in \overline{B_\delta 
(0)} \times \overline{B_\eta (0)} \times (I_1 \times U_1 \times P)$ and $| 
\partial_2 G\ee (t,x,(t_1,x_1,p)) | = O (| \log \ve |)$ for all $K \cptin I_1 
\times U_1 \times P$ and $(t,x,(t_1,x_1,p)) \in \overline{B_\delta (0)} \times 
\overline{B_\eta (0)} \times K$. By Theorem \ref{gen ode p}, there exists $v \in 
\calG [(I_1 \times U_1 \times P) \times \hat J, B_\eta (0)]$ with $\hat J:=(- 
\hat h, \hat h)$ and $\hat h = \min \left( \delta, \frac{\eta}{a} \right)$ such 
that for all $(\ti t_1,\ti x_1,\ti p) \in \ti I_{1c} \times \ti U_{1c} \times 
\ti P_c$ the map $v(\ti t_1,\ti x_1,\ti p,\, . \, ) \in \calG [\hat J, B_
 \eta (0)]$ is a solution of the initial value problem
\begin{equation}
\label{ivp anfangswert red}
\dot v (t) = G(t, v (t),(\ti t_1,\ti x_1,\ti p)), \quad v(0)=0 .
\end{equation}
The solution $v$ is unique in $\calG [(I_1 \times U_1 \times P) \times \hat 
J,\hat U]$.

By Remark \ref{rem gen ode p}, there exists a representative $(v\ee)\ee$ of $v$ 
that satisfies the classical initial value problem for all $(t_1,x_1,p) \in I_1 
\times U_1 \times P$ and $\ve$ sufficiently small. Let $\sigma \in \big[ \frac 1 
2,1 \big)$, $h := \sigma \hat h$ and $h_1:= \min ((1- \sigma ) \hat h, 
(1-\lambda ) \alpha )$. Set $J:= (t_0-h,t_0+h)$ and $J_1:= (t_0-h_1,t_0+h_1)$. 
Then $J_1 \subseteq J \subseteq \hat J$. We now define $u\ee : J_1 \times U_1 
\times P \times J \to \setR^n$ by
\[
u\ee(t_1,x_1,p,t):= v\ee(t_1,x_1,p,t-t_1) + x_1.
\]
The map $u\ee$ is well-defined since $J_1 \subseteq I_1$ and
\begin{equation}
\label{alles ok}
|t-t_1| \le |t-t_0|+|t_0-t_1|
\le h + h_1
\le \sigma \hat h + (1-\sigma) \hat h
= \hat h.
\end{equation}
The moderateness of $(u\ee)\ee$ is an immediate consequence of the moderateness 
of $(v\ee)\ee$. By \eqref{alles ok} and since $x_1-x_0 \in B_{\mu}(0)$ for all 
$x_1 \in U_1$, it follows that
\begin{align*}
u\ee(J_1 \times U_1 \times P \times J)
& \subseteq v\ee(I_1 \times U_1 \times P \times \hat J) +x_1
\subseteq B_\eta (0) +x_1 \\
& \subseteq B_{\eta}(x_0) -x_0 +x_1
\subseteq \overline{B_{\eta}(x_0) + B_{\mu}(0)}
\subseteq B_\gamma (x_0),
\end{align*}
i.e., $u:=[(u\ee)\ee]$ is an element of $\calG [J_1 \times U_1 \times P \times 
J,B_\gamma (x_0)]$. Furthermore, the function $u\ee(\tee t,\tee x,\ti p\ee,\, . 
\, )$ satisfies
\begin{align*}
\frac{\partial}{\partial t} u\ee & (\tee t,\tee x,\ti p\ee,t)
= \frac{\partial}{\partial t} \Big( v\ee(\tee t,\tee x,\ti p\ee,t-\tee t) + 
\tee x \Big) \\
& = G\ee (t-\tee t, v\ee(\tee t,\tee x,\ti p\ee,t-\tee t), (\tee t,\tee x,\ti 
p\ee)) \\
& = F\ee (t, v\ee(\tee t,\tee x,\ti p\ee,t-\tee t) + \tee x,\ti p\ee)
= F\ee (t,u\ee(\tee t,\tee x,\ti p\ee,t),\ti p\ee)
\end{align*}
and
\[
u\ee(\tee t,\tee x,\ti p\ee,\tee t)= v\ee(\tee t,\tee x,\ti p\ee,0)+\tee x=\tee x
\]
for all $(\ti t_1,\ti x_1,\ti p)=([(\tee t)\ee],[(\tee x)\ee],[(\ti p\ee)\ee]) 
\in \ti J_{1c} \times \ti U_{1c} \times \ti P_c$ and $t \in J$. Thus, $u(\ti 
t_1,\ti x_1,\ti p,\, . \,)$ is indeed a solution of the initial value problem 
\eqref{ivp anfangswert}.

Note that for any $h \in \left( 0 , \min \big( \alpha, \frac{\beta}{a} \big) 
\right)$ the constants $\lambda$, $\mu$, $\delta$, $\eta$, $\hat h$ and $\sigma$ 
can be chosen within their required bounds such that all the necessary 
inequalities in the construction of $(u\ee)\ee$ are satisfied.

{\it Uniqueness}:  By Proposition \ref{punktweise}, it suffices to show that for 
every near-standard point $(\ti t_1,\ti x_1,\ti p) = ([(\tee t)\ee],[(\tee 
x)\ee],\ti p) \in \ti J_{1c} \times \ti U_{1c} \times \ti P_c$ the solution 
$u(\ti t_1,\ti x_1,\ti p,\, . \,)$ is unique in $\calG [J,B_\gamma (x_0)]$: Let 
$(\tee t,\tee x) \to (t_1,x_1) \in J_1 \times U_1$ for $\ve \to 0$. Assume that 
$w(\ti t_1,\ti x_1,\ti p) \in \calG [J,B_\gamma (x_0)]$ is another solution of 
(\ref{ivp anfangswert}). For brevity's sake we simply write $u$ and $w$ in place 
of $u(\ti t_1,\ti x_1,\ti p)$ and $w(\ti t_1,\ti x_1,\ti p)$, respectively.

We will show that $w|_{(t_0-r,t_0+r)}=u|_{(t_0-r,t_0+r)}$ holds for any $r \in 
(0,h)$. Since $\calG$ is a sheaf, the equality of $w$ and $u$ then also holds on 
$J$.

Now, let $r \in (0,h)$ and set $\rho := \frac 1 2 (h-r)$. Define $\bar w: 
B_{r+\rho}(t_0-t_1) \to B_{\gamma+\mu}(0)$ by $\bar w(t):=w(t+\ti t_1)-\ti x_1$. 
From $\ti t_{1\ve} \to t_1$ as $\ve \to 0$ it follows that $\bar w$ is 
well-defined. Then, by the choice of $\rho$ and Proposition \ref{composition}, 
$\bar w \in \calG[B_{r+\rho}(t_0-t_1), B_{\gamma+\mu}(0)]$. Moreover, $\bar w$ 
is a solution of the initial value problem \eqref{ivp anfangswert red}. Since 
$B_{r+\rho} (t_0-t_1) \subseteq \hat J$ and solutions of \eqref{ivp anfangswert 
red} are unique in $\calG [\hat J,B_{\gamma+\mu} (0)]$, it follows that $\bar 
w=v(\ti t_1,\ti x_1,\ti p,\,.\,)|_{B_{r+\rho} (t_0-t_1)}$. Noting that
\[
w(t)
= \bar w(t- \ti t_1)+ \ti x_1
= v(\ti t_1,\ti x_1,\ti p,t- \ti t_1)+ \ti x_1
= u(t),
\]
we finally arrive at $w|_{(t_0-r,t_0+r)}=u|_{(t_0-r,t_0+r)}$.
\end{proof} 

\begin{remark}
Concerning representatives, a remark analogous to \ref{rem gen ode p} also 
applies to Theorem \ref{gen ode anfang}.
\end{remark}


\section{A Frobenius theorem in generalized functions}

\label{frob}

In this section, we will use the following notation: By $\setR^{m \times n}$ we 
denote the space $\setR^{mn}$, viewed as the space of $(m \times n)$-matrices 
over $\setR$. A similar convention applies to $\calR^{m \times n}$ and $\calG 
(U)^{m \times n}$. For any $u \in \calG (U)^m$ the derivative $\Diff u$ can be 
regarded as an element of $\calG (U)^{m \times n}$.

Now we are ready to prove a generalized version of the Frobenius Theorem.

\begin{theorem}
\label{gen frob}
Let $U$ be an open subset of $\setR^n$, $V$ an open subset of $\setR^m$ and $F 
\in \calG (U \times V)^{m \times n}$. Let $\alpha >0$ be chosen such that 
$\overline{B_\al(x_0)} \cptin U$. Let $(\toe{y})\ee$ be a representative of $\ti 
y_0$ and $L \cptin V$, $\ve_0 \in (0,1]$ such that $\toe{y} \in L$ for all $\ve 
\le \ve_0$. With $\beta>0$ satisfying $L_\beta:=L+\overline{B_\beta (0)} \cptin 
V$ set
\[
Q:=\overline{B_\alpha (x_0)} \times L_\beta
\qquad (\cptin U \times V).
\]
Assume that $F$ has a representative $(F\ee)\ee$ satisfying
\begin{equation}
\label{unif bded}
\sup_{(x,y) \in Q} |F\ee (x,y)| \le a
\qquad (\ve \le \ve_0)
\end{equation}
for some constant $a>0$ and
\begin{equation}
\label{log bed}
\sup_{(x,y) \in Q} |\partial_2 F\ee(x,y)| = O(| \log \ve |).
\end{equation}

Then the following are equivalent:
\begin{enumerate}[(A)]
\item
\label{frob eins}
For all $(\ti x_0, \ti y_0) \in \ti U_c \times \ti V_c$ with $\ti x_0 \approx 
x_0 \in U$ the initial value problem
\begin{equation}
\label{ivp total}
\Diff u(x)=F(x,u(x)), \quad u( \ti x_0)= \ti y_0
\end{equation}
has a unique solution $u(\ti x_0,\ti y_0)$ in $\calG [U(\ti x_0,\ti y_0),W]$, 
where $U(\ti x_0,\ti y_0)$ is an open neighbourhood of $x_0$ in $U$ and 
$W=L+B_\beta (0)$.
\item
\label{frob zwei}
The integrability condition is satisfied, i.e., the mapping
\begin{align}
\label{int cond}
(x,y,v_1,v_2) & \mapsto \Diff F(x,y) (v_1, F(x,y) (v_1)) (v_2)
\end{align}
is symmetric in $v_1, v_2 \in \setR^n$ as a generalized function in $\calG (U 
\times V \times \setR^n \times \setR^n)^m$.
\end{enumerate}
\end{theorem}

\begin{proof} We follow the line of argument of the classical proof based on 
the ODE theorem with parameters.

\eqref{frob eins} $\Rightarrow$ \eqref{frob zwei}:
By Proposition \ref{punktweise}, we only have to check the integrability 
condition \eqref{int cond} for all near-standard points $\ti v_1,\ti v_2 \in \ti 
\setR_c^n$ and $(\ti x, \ti y) \in \ti U_c \times \ti V_c$: By \eqref{frob 
eins}, there exists a solution $u$ of the initial value problem
$
\Diff u(x)=F(x,u(x)), 
u( \ti x)= \ti y.
$
Writing $\Diff u$ as $\Diff u =F \circ (\id,u)$, we obtain
\begin{align*}
\Diff^2 & u(\ti x)(\ti v_1,\ti v_2)
= \big(\Diff^2 u(\ti x) (\ti v_1)\big) (\ti v_2)
= \big( \Diff (F \circ (\id,u))(\ti x) (\ti v_1) \big) (\ti v_2) \\
& = \Big( \big( \Diff F(\ti x,u(\ti x)) \circ (\id, \Diff u(\ti x)) \big) (\ti 
v_1) \Big) (\ti v_2) \\
& = \Big( \Diff F \big( \ti x,u(\ti x) \big) \big( \ti v_1,F(\ti x,u(\ti x)) 
(\ti v_1) \big) \Big) (\ti v_2)
= \Diff F(\ti x, \ti y) \big( \ti v_1,F(\ti x, \ti y) (\ti v_1) \big) (\ti v_2)
\end{align*}
for all near-standard points $\ti v_1, \ti v_2 \in \ti \setR_c^n$. The last 
expression is symmetric in $\ti v_1$ and $\ti v_2$ since, by Schwarz's Theorem, 
$\Diff^2 u(\ti x)$ has this property.

\eqref{frob zwei} $\Rightarrow$ \eqref{frob eins}:
Let $\ti x_0=[(\toe{x})\ee]$ be a near-standard point in $\ti U_c$ with $\ti 
x_0 \approx x_0$ and let $\ti y_0 \in \ti V_c$.

{\it Existence}: Choose $\delta \in (0,\alpha)$ and set $\gamma:=\alpha 
-\delta$. We can assume without loss of generality that $\toe{x} \in B_\delta 
(x_0)$ for all $\ve \le \ve_0$. Then, for $t \in (-\gamma,\gamma)$ and $v \in 
B_1(0) \subseteq \setR^n$, we have  $\toe{x} + tv \in B_\alpha (x_0) \subseteq 
U$ and, thus, the function
\[
\begin{array}{cccc}
G\ee : & ( -\gamma,\gamma ) \times V \times B_1(0) & \to & \setR^m \\
& (t,y,v) & \mapsto & F\ee (\toe{x} +tv,y) (v)
\end{array}
\]
is well-defined. By Proposition \ref{composition}, $G:=[(G\ee)\ee]$ is a 
well-defined generalized function in $\calG\left( ( -\gamma,\gamma ) \times V 
\times B_1(0) \right)^m$.
Now consider the initial value problem
\begin{equation}
\label{ivp frob}
\dot f (t) = G(t,f(t),v), \quad f(0)= \ti y_0,
\end{equation}
with parameter $v \in B_1(0)$. Then the conditions of Theorem \ref{gen ode p} 
are satisfied, i.e., 
\[
| G\ee (t,y,v) | \le a
\qquad \mbox{and} \qquad
\partial_2 G\ee (t,y,v) = O( | \log \ve |)
\]
for all $(t,y,v) \in \overline{B_\eta (0)} \times L_\beta \times B_1(0)$ with 
$\eta \in (0,\gamma)$ fixed. From Theorem \ref{gen ode p}, it follows that there 
exists a generalized function $f \in \calG [B_1(0) \times J,W]$ with 
$J:=[-h,h]$, $h:=\min \big( \eta,\frac{\beta}{a} \big)$ and $W:=L+B_\beta(0)$ 
such that $f(v,\, . \, )$ is a solution of \eqref{ivp frob} for all $v \in 
B_1(0)$. Fix some $r \in (0,h)$ and $\lambda \in (0,1)$ and set
\[
U(\ti x_0,\ti y_0)
:= B_{\lambda r} (x_0).
\]
Assuming without loss of generality that $|x_0 - \toe{x}|< (1- \lambda )r$ for 
all $\ve \le \ve_0$, the function $u\ee (\ti x_0,\ti y_0) : U(\ti x_0,\ti y_0) 
\to W$ given by
\[
u\ee (\ti x_0,\ti y_0)(x):= f\ee \left( \frac{1}{r} (x- \toe{x}),r \right)
\]
is well-defined. \kern-.9pt By Proposition \ref{composition}, $u(\ti x_0,\ti 
y_0):=[(u\ee(\ti x_0,\ti y_0))\ee]\kern-.9pt \in\kern-.9pt \calG [U(\ti x_0, \ti 
y_0),W]$. From now on, we will denote $u(\ti x_0,\ti y_0)$ simply by $u$.

The fact that $u$ is indeed a solution of \eqref{ivp total} follows from 
\begin{equation}
\label{gleich!}
\partial_1 f(v,t) (w)
= F(x_0+tv,f(v,t)) (t w)
\quad\mbox{in $\calG ((-h,h) \times B_1(0) \times \setR^n)^m$.}
\end{equation}
Assuming this to be true for the moment, we have
\begin{align*}
\Diff u(x) (\ti w)
&= \Big( \frac{\partial}{\partial x} f \Big( \frac{x-\ti x_0}{r},r \Big) \Big) 
(\ti w)
= \partial_1 f \Big( \frac{x-\ti x_0}{r},r \Big) \Big(\frac{1}{r} \ti w \Big) \\
& = F(x,u(x)) (\ti w)
\end{align*}
for all $\ti w \in \ti \setR_c^n$. Applying Proposition \ref{punktweise} to the 
above equation, we obtain $\Diff u(x)=F(x,u(x))$ in $\calG [U(\ti x_0,\ti 
y_0),W]$. Moreover, we observe that $f(0,\, . \, )$ is the (in $\calG 
[(-h,h),W]$) constant function $t \mapsto \ti y_0$, and hence we obtain
$
u(\ti x_0)
= f ( \frac 1 r (\ti x_0 - \ti x_0),r )
= \ti y_0.
$
Thus, $u$ is indeed a solution of the initial value problem \eqref{ivp total}.

To complete the proof of existence, it remains to show \eqref{gleich!}:
Consider the net $(k\ee)\ee$ given by $k\ee: (-h,h) \times B_1(0) \times 
\setR^n \rightarrow \setR^m$,
\[
k\ee(t,v,w):=\partial_1 f\ee(v,t) (w) - F\ee(\toe{x}+tv,f\ee(v,t)) (tw).
\]
Note that, by Proposition \ref{composition}, $k:=[(k\ee)\ee]$ is a well-defined 
generalized function in $\calG ((-h,h) \times B_1(0) \times \setR^n)^m$. Let 
$\ti v \in \widetilde{B_1(0)}_c$ and $\ti w \in \ti \setR^n_c$. Differentiating 
$k(t, \ti v, \ti w)$ with respect to $t$, using the fact that $f(\ti v,\, . \,)$ 
is a solution of \eqref{ivp frob} and setting $\ti z=(\ti x_0+t \ti v,f(\ti 
v,t))$, we obtain
\[
\dot k(t,\ti v,\ti w)
= \partial_1 F(\ti z) (t \ti w, \ti v) + \partial_2 F(\ti z) (\partial_1 f(\ti 
v,t)(\ti w), \ti v) - \Diff F(\ti z) (\ti v,F(\ti z)  (\ti v)) (t \ti w).
\]
Applying the integrability condition \eqref{frob zwei} to the last term on the 
right-hand side, we arrive at
\begin{equation}
\label{linear ivp}
\dot k(t,\ti v,\ti w)
= \Big( \partial_2 F(\ti x_0+t \ti v,f(\ti v,t))\flip (\ti v) \Big)  \cdot 
k(t,\ti v,\ti w).
\end{equation}
Moreover, observe that $k(0,\ti v,\ti w)=0$ in $\ti \setR^m$. Hence, $k(\, . \, 
,\ti v,\ti w)$ is a solution of a linear initial value problem. Setting $A_{\ti 
v}(t):= \partial_2 F(\ti x_0+t \ti v,f(\ti v,t))\flip (\ti v)$, it follows from 
\eqref{log bed} that 
\[
\sup_{t \in (-h,h)} | A_{\ti v}(t) | = O(| \log \ve |).
\]
By a Gronwall argument similar to the one in the uniqueness proof of Theorem 
\ref{gen ode} we infer that $k(\, . \, ,\ti v,\ti w)=0$ is the only solution of 
\eqref{linear ivp}. By Proposition \ref{punktweise}, we conclude that $k=0$ in 
$\calG ((-h,h) \times B_1(0) \times \setR^n)^m$, thereby establishing the claim.

{\it Uniqueness}: Let $\bar u \in \calG [B_{\lambda r}(x_0),W]$ be another 
solution of \eqref{ivp total}. We will show that $\bar 
u|_{B_s(x_0)}=u|_{B_s(x_0)}$ for all $s<\lambda r$. Since $\calG$ is a sheaf, 
the equality then also holds on $B_{\lambda r}(x_0)=U(\ti x_0,\ti y_0)$.

Let $s \in (0,\lambda r)$ and let $\ti v=[(\ti v\ee)\ee] \in 
\widetilde{B_1(0)}_c$. Setting $\sigma :=\frac 1 3 (\lambda r - s)$, we define 
$g(\ti v, \,.\,) : (-s-2\sigma,s+2\sigma) \to W$ by $g(\ti v ,t):=\bar u (\ti 
x_0 + t \ti v)$. From $\ti x_{0\ve} \to x_0$ as $\ve \to 0$ it follows that 
$g(\ti v,\,.\,)$ is well-defined. Then, by the choice of $\sigma$ and by 
Proposition \ref{composition}, $g(\ti v, \,.\,) \in 
\calG[(-s-2\sigma,s+2\sigma),W]$. Moreover, $g(\ti v,\,.\,)$ is a solution of 
\eqref{ivp frob} for $v=\ti v$. Since $(-s-2\sigma,s+2\sigma) \subseteq J$ and 
solutions of \eqref{ivp frob} are unique in $\calG [J,W]$, it follows that 
$g(\ti v,\,.\,)=f(\ti v,\, . \, )|_{(-s-2\sigma,s+2\sigma)}$ for all $\ti v \in 
\widetilde{B_1(0)}_c$. By Proposition \ref{punktweise}, $g: (v,t) \mapsto 
g(v,t)$ is equal to $f$ on $(-s-2\sigma,s+2\sigma)$. Observe that for 
$c_1,c_2>0$ the generalized functions
$
(v,t) \mapsto f\Big(\frac{1}{c_1}v,c_1 t\Big)
$
and
$
(v,t) \mapsto f\Big(\frac{1}{c_2}v,c_2 t\Big)
$
are equal on the intersection of their domains. Hence, we obtain
\begin{align*}
\bar u(x)
& = g \Big( \frac{1}{s+\sigma} (x-\ti x_0) \Big) (s+\sigma)
= f \Big( \frac{1}{s+\sigma} (x-\ti x_0) \Big) (s+\sigma) \\
& = f \Big( \frac{1}{r} (x-\ti x_0) \Big) (r)
= u(x),
\end{align*}
thereby establishing the claim.
\end{proof}

\section*{Acknowledgement}
The research was funded by the Austrian Science Fund (FWF): P23714-N13.



\end{document}